\documentclass[12pt]{article}

\usepackage{tikz}
\usetikzlibrary{snakes}
\usepackage{amsmath,amssymb,amsthm,amscd}
\usepackage{multirow}
\usepackage{mathtools}
\usepackage{enumitem}
\usepackage{pgf,tikz}
\usetikzlibrary{arrows}

\numberwithin{equation}{section}

\newtheorem{theorem}{Theorem}[section]
\newtheorem{proposition}[theorem]{Proposition}

\newtheorem*{theorem*}{Theorem}

\theoremstyle{definition}

\newtheorem{example}[theorem]{Example}
\newtheorem{remark}[theorem]{Remark}

\newcommand{\PP}{ \ensuremath{\mathbb{P}}}

\newcommand{\mL}{\mathcal{L}}

\renewcommand\P{\mathbb P}

\newcommand\GF{K}

\begin{document}

\author{M.~Dumnicki, {\L}.~Farnik,  B.~Harbourne,\\ T.~Szemberg, H.~Tutaj-Gasi\'nska}

\title{ Veneroni maps}

\maketitle
\begin{abstract}
   Veneroni maps are a class of birational transformations of projective spaces. This class contains
   the classical Cremona transformation of the plane, the cubo-cubic transformation of the space
   and the quatro-quartic transformation of $\P^4$. Their common feature is that they are determined by
   linear systems of forms of degree $n$ vanishing along $n+1$ general flats of codimension $2$ in $\P^n$.
   They have appeared recently in a work devoted to the so called unexpected hypersurfaces. The purpose of this
   work is to refresh the collective memory of the mathematical community about these somewhat forgotten
   transformations and to provide an elementary description of their basic properties given from a modern
   point of view.
\\
\\
\footnotesize{Keywords: Cremona transformation, birational transformation
\\	
	}
\end{abstract}

\section{Introduction}
The aim of this note is to give a detailed description of Veneroni's Cremona transformations in $\P^n$.
They were first described by Veneroni in \cite{veneroni}, and then discussed for $n=4$ by Todd in \cite{Todd1930} and by Blanch in \cite{Blanch}
and for $n\geq3$ by Snyder and Rusk in \cite{Snyder2} (with a focus on $n=5$) and by Blanch again in \cite{Blanch2}.
The base loci of the Veneroni transformations involve certain varieties swept by lines
that were considered for $n=4$ by Segre in \cite{segre} and for $n\geq3$ by Eisland in \cite{eiesland}.
Evolution in terminology and rigor can make it a challenge to study classical papers. Our purpose here is
to bring this work together in one place, in a form accessible to a modern audience.
In order to use Bertini's Theorem, we assume the ground field $\GF$ has characteristic 0.

Consider $n+1$ distinct linear subspaces $\Pi_0,\dots,\Pi_{n}\subset\P^n$ of codimension $2$.
Let $\mL_n$ be the linear system of hypersurfaces in $\P^n$ of degree $n$ containing
$\Pi_0\cup\dots\cup\Pi_{n}$ and let $N+1$ be the vector space dimension of $\mL_n$
(we will see that $N=n$ when the $\Pi_j$ are general, hence by semi-continuity we have $N\geq n$).
We denote by $v_n:\P^n\dashrightarrow \P^N$ the rational map given by $\mL_n$.
If $N=n$ and if in addition $v_n$ is birational, we refer to $v_n$ as a Veneroni transformation.
(Here we raise an interesting question: is $v_n$ birational to its image if and only if $N=n$?
When $n=2$, it is not hard to check that $N=n$ always holds and that $v_n$ is always birational.)

When the $\Pi_j$ are general, we will see that $v_n$ is a Veneroni transformation
whose inverse is also given by a linear system of forms of degree $n$
vanishing on $n+1$ codimension 2 linear subspaces of $\P^n$. In this situation,
$v_2$ is the standard quadratic Cremona transformation of $\mathbb{P}^2$,
$v_3$ is a cubo-cubic Cremona transformation of $\mathbb{P}^3$ (see \cite[Example 3.4.3]{dolg}) and
$v_4$ is a quarto-quartic Cremona transformation of $\mathbb{P}^4$ (see \cite{Todd1930}). In \cite{H-M-TG}
the quarto-quartic Cremona transformation was used to produce some unexpected hypersurfaces.

The paper is organized as follows: we start in Section 2 with characterizing degree $n-1$ hypersurfaces in $\P^n$, containing
$n$ general linear subspaces of codimension $2$.

In Section 3 we investigate the linear system giving the Veneroni transformation.
When the spaces $\Pi_i$ are general, we prove that the dimension of $\mL_n$ is $n+1$,
we describe the base locus of this system, and prove that $v_n$ is birational.

In Section 4 we give the inverse $u_n$ of $v_n$ explicitly and show that $u_n$ is given by a
possibly sublinear system of the linear system
of forms of degree $n$ vanishing on a certain set of $n+1$ codimension 2 linear subspaces.

The last section, Section 5, is devoted to the additional description of the intersection of two
hypersurfaces of the type described in Section 2.

\section{Codimension $2$ linear subspaces}

Given linear subspaces $\Lambda_1, \dots, \Lambda_s$ of $\P^n$, a line intersecting them
all is called a \emph{transversal} (for  $\Lambda_1, \dots, \Lambda_s$).

\begin{proposition}\label{lem1}
Let $\Pi_1,\dots,\Pi_{n-1}$ be general codimension $2$ linear subspaces of $\P^n$.
For every point $p\in\P^n$, there is
a transversal for $\Pi_1,\dots,\Pi_{n-1}$ through $p$.
If $p$ is general, then there is a unique transversal, which we denote $t_p$,
and it meets $\Pi_1\cup\dots\cup\Pi_{n-1}$ in $n-1$ distinct points.
If however there are at least two transversals through $p$,
then $p$ lies on a subspace $T_p$ (of dimension $d_p>1$) intersecting each $\Pi_j$ along a subspace of dimension $d_p-1$, $j=1,\dots,n-1$,
and  $T_p$ is the union of all  transversals for $\Pi_1,\dots,\Pi_{n-1}$ through $p$.
\end{proposition}

\begin{proof}
Let $H$ be a general hyperplane in $\P^n$ and consider the projection $\pi_p:\P^n\dashrightarrow H$ from $p\in\P^n$.
If $p\not\in \Pi_j$, let $\Pi_j'=\pi_p(\Pi_j)$ and define
$$\Pi'=\bigcap_{{1\leq j<n}\atop{p\not\in\Pi_j}}\Pi_j'.$$
The intersection $\Pi'$ is not empty, since each $\Pi_j'$ is a hyperplane in $H$ and $\Pi'$
is the intersection of at most $n-1$ hyperplanes in $H$. Let $q\in\Pi'$.
Then the line $L_{pq}$ is transversal to all $\Pi_i$ (because either $q\in\Pi_i'$, and hence
$L_{pq}$ intersects $\Pi_i$, or $p\in\Pi_i$).
Conversely, a transversal from $p$ intersects $\Pi'$.
   Observe that for a general $p$,  the points $\pi^{-1}(q)|_{\Pi_j}$ are different, so the transversal meets $\Pi_j$ in different points.

Consequently, for a general $p$ there is a unique transversal. If dim $\Pi'=k>0$, then we have a subspace $T_p$ of the transversals  of dimension $k+1$. This subspace is a cone over $\Pi'$ and over $\Pi_j\cap T_p$ as well, hence dim $\Pi_j\cap T_p=k$.

\end{proof}

\begin{example}
For 3 general codimension 2 linear subspaces $\Pi_1,\Pi_2,\Pi_3$ of $\P^4$,
the pairwise intersections $\Pi_{ij}=\Pi_i\cap\Pi_j$, $i\neq j$, are points. These
three points span a plane $T$ which intersects each $\Pi_i$ in a line.
(For $\Pi_1$ this line is the line $L_{23}$ through $\Pi_{12}$ and $\Pi_{13}$, and similarly for $\Pi_2$ and $\Pi_3$.)
The lines $L_{12}, L_{13}, L_{23}$ all lie in $T$, hence every point $p\in T$ has a pencil of transversals, namely the lines
in $T$ through $p$.
\end{example}

\begin{remark}\label{coordsRemark}
We will eventually be interested in $n+1$ general codimension 2 subspaces $\Pi_0,\dots,\Pi_n$ of $\P^n$.
They are defined by $2(n+1)$ general linear forms $f_{j1},f_{j2}$, $j=0,\dots,n$, where
$I_{\Pi_j}=(f_{j1},f_{j_2})$. After a change of coordinates we may assume that $f_{j1}=x_j$ and that
$f_{j2}=a_{j0}x_0+\dots+a_{jn}x_n$ with $a_{ji}=0$ if and only if $i=j$. Here the homogeneous coordinate ring $R$ of $\P^n$ is
the polynomial ring $R=\GF[\P^n]=\GF[x_0,\dots,x_n]$.
\end{remark}

Now, we establish existence and uniqueness of a divisor $Q$ of degree $n-1$ containing $n$ general codimension $2$
linear subspaces in $\P^n$ for $n\geq 2$.


\begin{proposition}\label{uniqQ}
Let $\Pi_1,\dots,\Pi_{n}$ be general codimension $2$ linear subspaces of $\P^n$.
Then there exists a unique divisor $Q$ of degree $n-1$ containing $\Pi_j$ for $j=1,\dots,n$.
Moreover, $Q$ is reduced and irreducible, it is the union of the transversals for
$\Pi_1,\dots,\Pi_{n}$, and for each point $q\in Q$ we have
$\operatorname{mult}_qQ\geq r$, where $r$ is the number of indices $i$ such that $q\in\Pi_i$.
If $q$ is a general point of $Q$, then there is a unique transversal for
$\Pi_1,\dots,\Pi_{n}$ through $q$.
\end{proposition}

\begin{proof}
Let $\Delta$ be the determinantal variety in $(\P^n)^{n+1}$ of all $(n+1)\times(n+1)$ matrices $M$ of rank at most 2
whose entries are the variables $x_{ij}$. It is known that $\Delta$ is reduced and irreducible of dimension $3n-1$,
see \cite{HocEag71}.
It consists of the locus of points $(p_1,\dots,p_{n+1})$ whose span in $\P^n$ is contained in a line.

Let $\pi_i:(\P^n)^{n+1}\to\P^n$ be projection to the $i$th factor (so $1\leq i\leq n+1$).
Let $\Pi_i'=\pi_i^{-1}(\Pi_i)$. Then $D=\Delta\cap\bigcap_{1\leq i\leq n}\Pi_i'$ has dimension $3n-1-2n=n-1$, and
by repeatedly applying Bertini's Theorem, we see that $D$ is reduced and irreducible. Since $\Pi_1\cap\dots\cap\Pi_n=\varnothing$,
we see that $D$ is the locus
of all points $(p_1,\dots,p_{n+1})$ such that the span $\langle p_1,\dots,p_n\rangle$ is a line with
$p_i\in\Pi_i$ for $1\leq i\leq n$ and $p_{n+1}$ being on that line.
Thus $\overline{D}=\pi_{n+1}(D)$ is irreducible, properly contains $\Pi_1\cup\dots\cup\Pi_{n}$ and is
the union of all transversals for $\Pi_1,\dots,\Pi_{n}$.
In particular, $\overline{D}$ has dimension $n-1$, and since by Proposition \ref{lem1} there is a line
through a general point meeting $n-1$ of the spaces $\Pi_i$ in distinct points, we see that $\deg \overline{D}\geq n-1$.

Below we will check that there is a hypersurface $Q$ of degree $n-1$ containing $\Pi_1\cup\dots\cup\Pi_{n}$.
Since any such hypersurface must by Bezout contain all transversals for $\Pi_1,\dots,\Pi_{n}$,
we see that $\deg \overline{D} = n-1$ and $Q=\overline{D}$ and thus that there is a unique hypersurface of degree $n-1$ containing
$\Pi_1\cup\dots\cup\Pi_{n}$, and it is irreducible.

To show existence of $Q$ we follow \cite{Blanch2}. As mentioned in Remark \ref{coordsRemark},
we may assume that the ideal of $\Pi_k$ is
$$I_k=(x_k, f_k=a_{k,0}x_0+\dots +a_{k,k-1}x_{k-1}+a_{k,k+1}x_{k+1}+\dots +a_{k,n}x_n),$$
where we write $f_k$ instead of $f_{k,2}$.
By generality we may assume that $a_{i,j}\neq0$ for all $i\neq j$.

Now consider the $n\times n$ matrix
$$A=\left(\begin{array}{ccccc}
-f_1 & \dots & a_{1,k}x_k & \dots & a_{1,n}x_n \\
\vdots & & \vdots & & \vdots \\
a_{k,1}x_1 & \dots & -f_k & \dots & a_{k,n}x_n \\
\vdots & & \vdots & & \vdots \\
a_{n,1}x_1 & \dots & a_{n,k}x_k &\dots & -f_n
\end{array}\right)$$
and let $F=\det(A)$. Note that $F$ is not identically 0 (since its value at
the point $(1,0,\dots,0)$ is not 0) so $\deg F=n$.
It is clear, developing $\det(A)$ with respect to the $k$-th column,
that $F\in I_k$ for every $k=1,\dots,n$. For each $k$, adding to the $k$th column of $A$ all of the other columns of $A$
gives a matrix $A_k$ whose entries in the $k$th column are nonzero scalar multiples of $x_0$; in particular,
$$A_k=\left(\begin{array}{ccccccc}
-f_1 & \dots & a_{1,k-1}x_{k-1} &-a_{1,0}x_0 & a_{1,k+1}x_{k+1} & \dots & a_{1,n}x_n \\
\vdots & & \vdots & \vdots & \vdots & & \vdots \\
a_{k-1,1}x_1 & \dots & -f_{k-1} &-a_{k-1,0}x_0 & a_{k-1,k+1}x_{k+1} & \dots & a_{k-1,n}x_n \\
a_{k,1}x_1 & \dots & a_{k,k-1}x_{k-1} &-a_{k,0}x_0 & a_{k,k+1}x_{k+1} & \dots & a_{k,n}x_n \\
a_{k+1,1}x_1 & \dots & a_{k+1,k-1}x_{k-1} &-a_{k+1,0}x_0 & -f_{k+1} & \dots & a_{k+1,n}x_n \\
\vdots & & \vdots & \vdots & \vdots & & \vdots \\
a_{n,1}x_1 & \dots & a_{n,k-1}x_{k-1} &-a_{n,0}x_0 & a_{n,k+1}x_{k+1} & \dots & -f_n
\end{array}\right).$$
so
$$A_1=\left(\begin{array}{cccccc}
-a_{1,0}x_0 & a_{1,2}x_2 &\dots & a_{1,k}x_k & \dots & a_{1,n}x_n \\
-a_{2,0}x_0 & -f_2 &\dots & a_{2,k}x_k & \dots & a_{2,n}x_n \\
\vdots & \vdots & & \vdots & & \vdots \\
-a_{k,0}x_0 & a_{k,2}x_2 & \dots & -f_k & \dots & a_{k,n}x_n \\
\vdots & \vdots & & \vdots & & \vdots \\
-a_{n,0}x_0 & a_{n,2}x_2 & \dots & a_{n,k}x_k &\dots & -f_n
\end{array}\right).$$
Thus $F=\det(A)=\det(A_k)=x_0\cdot G$ for some polynomial $G$. Since $x_0$ is not an element of any $I_k$, it follows, that $G\in I_k$
for $k=1,\dots,n$, hence $G$ vanishes on each of $\Pi_1,\dots,\Pi_n$. Since $\deg F=n$,
we have $\deg(G)=n-1$. Thus $G$ defines a hypersurface $Q$ of degree $n-1$ containing each $\Pi_i$.

Now consider a point $q\in Q$. The matrix $A_k$ will have $r$ columns which vanish
at $q$, where $r$ is the number of indices $i$ such that $q\in\Pi_i$. In particular, each entry in each such column
is in the ideal $I_q$. Thus $G=\det(A_k)/x_0\in I_q^r$ so $\operatorname{mult}_qQ\geq r$.

Finally assume $p$ is a general point of $\Pi_n$. Since $\Pi_n$ is general, $p$ is a general point of $\P^n$,
hence by Proposition \ref{lem1} there is a unique transversal $t_p$ for $\Pi_1,\dots,\Pi_{n-1}$ through $p$,
hence $t_p$ is also the unique transversal for $\Pi_1,\dots,\Pi_{n}$ through $p$. Thus there is
an open neighborhood $U$ of $p$ of points $q$ through each of which there is a unique transversal
$t_q$ for $\Pi_1,\dots,\Pi_{n-1}$, and for those points $q$ of $U\cap Q$, $t_q$ also meets $\Pi_n$,
hence for a general point $q\in Q$ there is a unique transversal
$t_q$ for $\Pi_1,\dots,\Pi_{n}$.
\end{proof}

\begin{remark}\label{coordVertsRemark}
Let $p_0,\dots,p_n$ be the coordinate vertices of $\P^n$ with respect to the variables $x_0,\dots,x_n$, so
$p_0=(1,0,\dots,0)$, $\dots$, $p_n=(0,\dots,0,1)$. We saw in the proof of Proposition \ref{uniqQ} that
$p_0\not\in Q$ (since $F\neq0$ at $p_0$). Let $A_k'$ be the matrix from the proof of
Proposition \ref{uniqQ} arising after dividing $x_0$ from column $k$ of $A_k$. Then $Q$ is defined by $\det(A_k')=0$
but $A_k'$ at $p_k$ is a matrix which, except for column $k$, is a diagonal matrix with nonzero entries on the
diagonal, and whose $k$th column has no zero entries. Thus $\det(A_k')\neq 0$ at $p_k$ so $p_k\not\in Q$.
In particular, none of the coordinate vertices is on $Q$.
\end{remark}

\section{The system $\mL_n$}\label{sect3}

Let us start with some notation. Assume $\Pi_0,\dots, \Pi_n\subset \P^n$  are general linear subspaces of codimension 2.
From the previous section it follows that for each subset $\Pi_0,\dots,\Pi_{j-1},\Pi_{j+1},\dots,\Pi_{n}$  of $n$ of them there is a unique hypersurface $Q_j$ of degree $n-1$
containing them. Depending on the context, we may also denote by $Q_j$  the form defining this hypersurface.
We may assume $I_{\Pi_i}=(x_i,f_i)$ where $f_j$ is as given in Remark \ref{coordsRemark}.
In this case we have the $(n+1)\times(n+1)$ matrix
$$B=\left(\begin{array}{cccccc}
-f_0 & a_{0,1}x_1 & \dots & a_{0,k}x_k & \dots & a_{0,n}x_n \\
a_{1,0}x_0 &-f_1 & \dots & a_{1,k}x_k & \dots & a_{1,n}x_n \\
\vdots &\vdots& & \vdots & & \vdots \\
a_{k,0}x_0 &a_{k,1}x_1 & \dots & -f_k & \dots & a_{k,n}x_n \\
\vdots & \vdots& & \vdots & & \vdots \\
a_{n,0}x_0 &a_{n,1}x_1 & \dots & a_{n,k}x_k &\dots & -f_n
\end{array}\right).$$
Let $B_i$ be the $n\times n$ submatrix obtained by deleting row $i$ and column $i$ of $B$
(where we have $i$ run from 0 to $n$).
The matrix $A$ in the proof of Proposition \ref{uniqQ} is thus $B_0$, and we have $\det(B_i)=x_iQ_i$.
The next result shows that $v_n$ is the map given by $(x_0,\dots,x_n)\mapsto (x_0Q_0,\dots,x_nQ_n)$.

\begin{proposition}\label{prop4}
The polynomials $x_iQ_i$, $i=0,\dots,n$, give a basis for $\mL_n$, hence $\dim \mL_n=n+1$, so $v_n$ is a
rational map to $\P^n$ whose image is not contained in a hyperplane.
\end{proposition}

\begin{proof}
By Remark \ref{coordVertsRemark}, no coordinate vertex $p_j$ is in $Q_i$ for any $i$.
But $x_iQ_i\in\mL_n$ for every $i$, and $(x_iQ_i)(p_j)\neq0$ if and only if $i=j$.
Thus the polynomials $x_iQ_i$ span a vector space of dimension at least $n+1$.

To show that these sections in fact give a basis,
we show that $\dim \mL_n = n+1$.
We proceed by induction (the proof that $\mL_2$
has three sections is clear, since three general
points impose independent conditions on forms of degree 2 on $\P^2$).
Let $A$ be a fixed hyperplane that contains  $\Pi_1$. There is, by Proposition \ref{uniqQ}, a unique section
of $\mL_n$ containing $A$, namely $AQ_1$. Moreover, the restrictions to $A$ of
sections $s_n$ of $\mL_n$ which do not contain $A$ give divisors $s_n\cap A$ of degree $n$,
containing $\Pi_1$, and containing $A\cap \Pi_j, j>1$,
So on $A$, the linear system of restrictions residual to $\Pi_1$ has degree $n-1$ and contains
the $n$ general subspaces $\Pi_i\cap A$, $i>1$, of codimension $2$.
From the inductive assumption this has dimension $n$, so $\dim \mL_n = n+1$.

We may also see the result from the exact sequence
$$0\to \mL_n\otimes {\cal O}_{\PP^n}(-A)\to \mL_n\to\mL_n|_A\to 0,$$
where $A$ is as above. Then, from the inductive assumption, the dimension of $\mL_n\to\mL_n|_A$ is $n$, and from
Proposition \ref{uniqQ} the dimension of $\mL_n\otimes {\cal O}_{\PP^n}(-A)$ (which is of degree $n-1$ passing through $n$ codimension $2$ subspaces in $A$) is $1$.
\end{proof}

Let $T_n$ be the closure of the union of all lines transversal to $\Pi_0,\dots, \Pi_n$,
and let $R_n =Q_0\cap\dots\cap Q_n$ and let $B_n$ be the base locus of
$v_n : \P^n \dashrightarrow \P^n$. We note that $T_n \subseteq R_n$, by Proposition \ref{uniqQ}.

\begin{proposition}\label{prop2} We have $B_n=\Pi_0\cup\dots\cup\Pi_{n}\cup R_n$.
\end{proposition}

\begin{proof}
Since $v_n$ is given by $(x_0,\dots,x_n)\mapsto (x_0Q_0,\dots,x_nQ_n)$,
the base locus consists of the common zeros of the $x_iQ_i$.
Clearly each $Q_i$ (and hence each $x_iQ_i$) vanishes on $B_n$.
But $Q_i$ vanishes on $\Pi_j$ for $j\neq i$ and $x_i$ vanishes on $\Pi_i$,
so each $x_iQ_i$ vanishes on $\Pi_0\cup\dots\cup\Pi_{n}$.
Thus $\Pi_0\cup\dots\cup\Pi_{n}\cup R_n\subseteq B_n$.

Conversely, since $v_n$ is given by the linear system $\mL_n$, which
is spanned by the forms $HQ_i$ where $H$ is linear and vanishes on $\Pi_i$,
we see $B_n$ is the zero locus of the collections of such forms $HQ_i$.
Now consider a point $p\in B_n$ not in $\Pi_0\cup\dots\cup\Pi_n$.
Then for each $i$, there is a form $HQ_i$ such that $H$ vanishes on $\Pi_i$
but not at $p$, and hence $Q_i$ vanishes at $p$.
Thus $p\in R_n$, so $B_n\subseteq \Pi_0\cup\dots\cup\Pi_{n}\cup R_n$.
\end{proof}

\begin{proposition}\label{prop3}
We have $\dim T_n=n-2$ for $n\geq 3$, and $T_n$ is irreducible for $n>3$.
\end{proposition}

\begin{proof}
Consider the Grassmannian $V$ of lines in $\P^n$ and the incidence variety $W=\{(v,p)\in V\times\P^n: p\in L_v\}$,
where $L_v$ is the line corresponding to a point $v\in V$. We also have the two projections $\pi_1:W\to V$ and $\pi_2:W\to \P^n$.
Then $V$ is an irreducible variety of dimension  $2(n-1)$ and degree $\frac{(2(n-1))!}{n!(n-1)!} $ embedded in $\P^N, N=\binom{n+1}{2}-1$, see \cite{GH}.
The condition of being incident to a codimension 2 linear space is given by a hyperplane in $\P^N$ (see p. 128 in \cite{CP}),
so the intersections of $V$ with $n+1$ general hyperplanes gives the locus $\rho_n$ in $V$ parametrizing the lines
comprising $T_n$ and $\pi_2(\pi_1^{-1}(\rho_n))=T_n$. Thus $\dim \rho_n=2(n-1)-(n+1)=n-3$, so $\dim \pi_1^{-1}(\rho_n)=n-2$, and
by Proposition \ref{lem1} the projection $\pi_2$ is generically injective on $\pi_1^{-1}(\rho_n)$ so we have $\dim T_n=n-2$.
Moreover, by Bertini's Theorem, $\rho_n$ (and hence $T_n$) is irreducible when $\dim \rho_n>0$.
\end{proof}

\begin{proposition}\label{Tn=Rn}
	With the notation as above we have $T_n=R_n$ in $\P^n$.
\end{proposition}
\begin{proof}
	
	Let us start with the following fact.
Let $L_0,\dots,L_k,L$ be lines through a common point $p$. Let $L$ belong to a
space generated by $L_0,\dots,L_k,$ let $\cal{ P}$ be a linear subspace, such that $p$ does not
lie on $\cal{ P}$. Let $L_j$ intersect $\cal{ P}$ at a point $l_j, j=0...k$. Then $L$ intersects $\cal{ P}$,
as the linear combination of a
projection of some vectors is a projection of the combination.

Now we can show that the intersection of all $Q_j$ lies in $T_n$, the union
of all transversals. Observe, that the oposite inclusion is obvious.

Take a point $p$ in all $Q_j$, but not in any $\Pi_j$. So for each $j$, there is $L_j$
through $p$, transversal to all $Q_i$ except $Q_j$. We have $n+1$ such lines in
$\P^n$, hence there must be a linear dependence among them.

Without loss of generality, let $L_0$ belong to the space generated by the
others. Then using above lemma for $\cal{ P}=Pi_0$ we get that $L_0$ intersects
$Pi_0$ (since $L_1,\dots,L_n$ intersects $Pi_0$), which finishes the proof.

If $p \in Pi_j$ for some $j$,  the proof is trivial.
\end{proof}

\begin{proposition}\label{prop5}
The Veneroni transformation $v_n: \P^n\dashrightarrow \P^n$ is injective
off of $Q_0\cup\dots\cup Q_n$, hence it is a Cremona transformation.
\end{proposition}

\begin{proof}
By Proposition \ref{prop4}, sections of the form $H_i'Q_i, i=0,\dots,n$,
span $\mL_n$, where $H_i'$ is a hyperplane containing $\Pi_i$. For a point $p$ not in $Q_0\cup\dots\cup Q_n$
(and hence not in $\Pi_0\cup\dots\cup\Pi_n$), $H_i'+Q_i$ vanishes at $p$
if and only if $H_i'$ does. Moreover,
there is for each $i$ a unique hyperplane $H_i$ containing $\Pi_i$ and $p$. If the intersection of all such $H_i$
is not exactly $p$, then the intersection $H_0\cap\dots\cap H_n$
is a positive dimensional linear space,
and any line through $p$ in this space intersects each $\Pi_i$ and hence is a transversal for $\Pi_0,\dots,\Pi_n$,
and so $p$, being on a transversal, is in $T_n\subseteq R_n\subseteq B_n$.
Thus $v_n$ is injective off of $Q_0\cup\dots\cup Q_n$.
\end{proof}

\section{An inverse for $v_n$}

It is of interest to determine an inverse for $v_n$, and to observe that the inverse is
again given by forms of degree $n$ vanishing on $n+1$ codimension 2 linear subspaces.
We explicitly define such a map $u_n$ and then check that it is an inverse
for $v_n:\P^n\dashrightarrow \P^n$.
If we regard $x_0,\dots,x_n$ as homogeneous coordinates on the source $\P^n$ and
$y_0,\dots,y_n$ as homogeneous coordinates on the target $\P^n$, then $v_n$ is defined
by the homomorphism $h$ on homogeneous coordinate rings given by $h:y_i= x_iQ_i=\det(B_i)$,
as we saw in \S \ref{sect3}.

To define $u_n$, we slightly modify matrix $B$ from \S \ref{sect3}
by replacing the diagonal entries $-f_i$ in $B$ by $-g_i$ (defined below) and by replacing
each entry $a_{i,j}x_j$ in $B$ by $a_{i,j}y_j$ to obtain a new matrix
$$C=\left(\begin{array}{cccccc}
-g_0 & a_{0,1}y_1 & \dots & a_{0,k}y_k & \dots & a_{0,n}y_n \\
a_{1,0}y_0 &-g_1 & \dots & a_{1,k}y_k & \dots & a_{1,n}y_n \\
\vdots &\vdots& & \vdots & & \vdots \\
a_{k,0}y_0 &a_{k,1}y_1 & \dots & -g_k & \dots & a_{k,n}y_n \\
\vdots & \vdots& & \vdots & & \vdots \\
a_{n,0}y_0 &a_{n,1}y_1 & \dots & a_{n,k}y_k &\dots & -g_n
\end{array}\right).$$

To define $g_i$, recall that since $f_iQ_i\in \mL_n$ for each $i$ and the forms $x_jQ_j$
give a basis for  $\mL_n$, we can for each $i$
and appropriate scalars $b_{i,j}$ write
$$f_iQ_i=b_{i,0}x_0Q_0+\dots+b_{i,n}x_nQ_n.$$
We define $g_i$ to be $g_i=b_{i,0}y_0+\dots+b_{i,n}y_n$, so
we see that $h(g_i)=f_iQ_i$.

As an aside we also note that $b_{i,j}=0$ if and only if $i=j$.
(To see this, recall by Remark \ref{coordVertsRemark} that
no $Q_j$ vanishes at any coordinate vertex $p_k$,
but $f_i$ vanishes at the coordinate vertex $p_j$ if and only if $i=j$. Thus, evaluating
$f_iQ_i=b_{i,0}x_0Q_0+\dots+b_{i,n}x_nQ_n$
at $p_i$ gives $0=b_{i,i}Q_i$, hence $b_{i,i}=0$, while evaluating at $p_j$ for $j\neq i$ gives
$0\neq b_{i,j}Q_j$, hence $b_{i,j}\neq0$.)

Let $C_i$ be the matrix obtained from $C$ by deleting row $i$ and column $i$.
Define a homomorphism $\lambda: \GF[x_0,\dots,x_n] \to\GF[y_0,\dots,y_n]$
by $\lambda(x_i)=\det(C_i)$.

The next result gives an inverse for $v_n$.

\begin{proposition}\label{inverse}
The homomorphism $\lambda$ defines a birational map $u_n:\P^n\dashrightarrow\P^n$ which is inverse to $v_n$.
\end{proposition}

\begin{proof} Note that applying $h$ to the entries of $C$ gives the matrix obtained from $BD$,
where $D$ is the diagonal matrix whose diagonal entries are $Q_0,\dots,Q_n$, from which it is easy to see that
$h(\det(C_i))=\det(B_i)Q_0\cdots Q_{i-1}Q_{i+1}\cdots Q_n=x_iQ_0\cdots Q_n$.

We now have $h\lambda(x_i)=h(\det(C_i))=x_iQ_0\cdots Q_n$, so
$u_nv_n=id_U$, where $U$ is the complement of $Q_0\cdots Q_n=0$.
Since $v_n$ is a Cremona transformation, so is $u_n$ and thus
$u_n$ is the inverse of $v_n$.
\end{proof}

\begin{remark}\label{VeneroniInverseRem}
We now confirm that the forms $\det(C_i)$ defining $u_n$
have degree $n$ and vanish on $n+1$ codimension 2 linear subspaces $\Pi_i^*\subset\P^n$.
That $\deg(\det(C_i))=n$ is clear, since $C_i$ is an $n\times n$ matrix of linear forms.

Consider the codimension two linear spaces defined by the ideals
$J_k=(y_k,g_k=b_{k,0}y_0+\dots+b_{k,n}y_n$.
Since the entries of column $k$ of $C$ are in the ideal $J_k$, it follows that
$\det(C_i)$ vanishes on $\Pi_j^*$ for $j\neq i$. It remains to check that
$\det(C_i)$ vanishes on $\Pi_i^*$.
But let $q\in Q_i$ be a point where $v_n$ is defined.
Note that $y_i(v_n(q))=h(y_i)(q)=x_iQ_i(q)=0$ and that
$g_i(v_n(q))=h(g_i)(q)=f_iQ_i(q)=0$. Thus
$v_n|_{Q_i}$ gives a rational map to $\Pi_i^*$ whose image is in the zero locus of
$\det(C_i)$ since $\det(C_i)(v_n(q))=(h(\det(C_i)))(q)=h\lambda(x_i)(q)=(x_iQ_0\cdots Q_n)(q)=0$.
Thus $\det(C_i)$ vanishing on $\Pi_i^*$ will follow if we show that
$v_n|_{Q_i}$ gives a dominant rational map
to $\Pi_i^*$. This in turn will follow if we show for a general $q\in Q_i$ that the fiber over
$v_n(q)$ has dimension 1 (since $Q_i$ as dimension $n-1$ and $\Pi_i^*$ has dimension $n-2$).
But the space of forms in $\mL_n$ vanishing on $q$ is spanned
by forms of the form $H_jQ_j$ where $H_j$ is a hyperplane
containing $q$ and $\Pi_j$. For a general point $q$,
since the $\Pi_j$ are general, the intersection of any $n-1$ of the $H_j$
with $j\neq i$ has dimension 1. Since the $\Pi_j$ are general,
the same is true for a general point $q\in Q_i$ except now,
since there is a transversal $t_q$ through $q$ for $\Pi_j$, $j\neq i$, we see that
$\cap_{j\neq i}H_j$ still has dimension 1 and is thus exactly $t_q$. Hence the locus of points
on which the forms in $\mL_n$ vanishing at $q$ vanish is exactly $t_q$.
Thus the fiber over $v_n(q)$ has dimension 1, as we wanted to show.

The question remains as to whether $u_n$ is itself a Veneroni transformation
whenever $v_n$ is. If we denote by $\mL_n^*$ the forms in $\GF[y_0,\dots,y_n]$
of degree $n$ vanishing on $\Pi_0^*\cup\dots\cup\Pi_n^*$, what we saw above is that
$u_n$ is defined by an $n+1$ dimensional linear system contained in $\mL_n^*$;
the issue is whether the linear system is all of $\mL_n^*$
(i.e., whether $\dim \mL_n^*=n+1$).

In any case, when $\Pi_0,\dots,\Pi_n$ are general, we now see that
$v_n$ gives a birational map $\P^n\dashrightarrow\P^n$
whose restriction to $Q_i$ gives a rational map to $\Pi_i^*$ for $i=0,\dots,n$
and the fiber of $Q_i$ over $\Pi_i^*$ generically has dimension 1.
It is convenient to denote the linear system of divisors of degree $n$ vanishing on
$\Pi_0\cup\dots\cup\Pi_n$ by $nH-\Pi_0-\dots-\Pi_n$.
Similarly, the linear system of divisors of degree $n-1$ vanishing
on $\Pi_j$ for $j\neq i$ is represented by
$(n-1)H-\Pi_0-\dots-\Pi_n+\Pi_i$.
Thus, if $H^*$ is the linear
system of divisors of degree 1 on the target $\P^n$ for $v_n$,
then $v_n$ pulls $H^*$ back to $nH-\Pi_0-\dots-\Pi_n$,
and it pulls $\Pi_i^*$ back to $Q_i$, represented by
$(n-1)H-\Pi_0-\dots-\Pi_n+\Pi_i$.
We can represent the pullback by a matrix map
$M_n:{\mathbb Z}^{n+1}\to{\mathbb Z}^{n+1}$ where
$$
M_n=\begin{pmatrix}
n&n-1&n-1&\dots&n-1\\
-1 & 0 &-1 &\dots&-1\\
-1&-1&0&\dots&-1\\
\vdots&\vdots&\vdots&&\vdots\\
-1&-1&-1&\dots&0
\end{pmatrix}.
$$

If in fact the spaces $\Pi_i^*$ can be taken to be sufficiently general,
then $\dim \mL_n^*=n+1$, and $u_n$ pulls $H$ back to
$nH^*-\Pi_0^*-\dots-\Pi_n^*$, and it pulls $\Pi_i$ back to
$(n-1)H^*-\Pi_0^*-\dots-\Pi_n^*+\Pi_i^*$, and hence is represented
by the same matrix $M_n$. Since $M_n^2$ corresponds to the pullback map for $u_nv_n$
and $u_nv_n$ is the identity (where defined), we would expect that $M^2_n=I_n$,
which is indeed the case.
\end{remark}

\section{Intersection of $Q_i$ and $Q_j$}

This section is devoted to investigating the intersections of $Q_i$ and $Q_j$,
assuming that $\Pi_0,\dots,\Pi_n$ are general linear subspaces of codimension 2.
These intersections were already treated in \cite{Snyder2} and in more detail than here,
but here we use more modern language.

Without loss of generality assume that $i=0,j=1$, so take $Q_0\cap Q_1$.
From the considerations above (Proposition \ref{prop2}) we may write
$$Q_0\cap Q_1=T_n\cup\Pi_2\cup\dots\cup \Pi_{n}\cup M_n$$
where $M_n$ is the closure of the complement of
$T_n\cup\Pi_2\cup\dots\cup \Pi_{n}$ in $Q_0\cap Q_1$.

\begin{proposition}
The complement of $T_n\cup\Pi_2\cup\dots\cup \Pi_{n}$
in $Q_0\cap Q_1$ is the set of all
points $q\in Q_0\cap Q_1$ through which there is no transversal for $\Pi_0,\dots,\Pi_n$,
(in which case there is more than one transversal through $q$ for $\Pi_2,\dots,\Pi_{n}$).
\end{proposition}

\begin{proof}
For $n=2$ it is easy to check that $Q_0\cap Q_1=\Pi_2$ and that $T_n=M_n=\varnothing$.
For $n=3$, keeping in mind that $Q_0=\P^1\times\P^1$,
$Q_0\cap Q_1$ is a divisor on $Q_0$ of multi-degree $(2,2)$, consisting of
the lines $\Pi_2$ and $\Pi_3$ together with the two transversals
for $\Pi_0,\dots,\Pi_3$ (these two transversals give $T_n$); again $M_n$ is empty. (See, for example, the description of the
cubo-cubic Cremona transformation from \cite{dolg} or \cite{DHRST-G}.)

So now assume that $n\geq 4$.
Take a point $q$ from $Q_0\cap Q_1$. Suppose $q$ is not in $\Pi_2\cup\dots\cup\Pi_n$.
Since $q\in Q_0$, by Proposition \ref{uniqQ} there is at least one transversal through
$q$ for $\Pi_1,\dots,\Pi_n$ and since $q\in Q_1$ there is similarly at least one transversal through
$q$ for $\Pi_0,\Pi_2\dots,\Pi_n$. If one of the transversals coming from $q\in Q_0$
is also a transversal coming from $q\in Q_1$, then it follows that the transversal
goes through all $\Pi_j$, so the transversal (and hence $q$) is contained in $T_n$.
Otherwise, $q\not\in T_n$, hence
there are two lines through $q$ transversal for $\Pi_2,\dots,\Pi_n$.
\end{proof}

\begin{example}
We close by showing for $n=4$ that the complement of $T_4\cup\Pi_2\cup\dots\cup \Pi_{4}$
in $Q_0\cap Q_1$ is nonempty.

Take three points $p_{ij}$, where $p_{ij}=\Pi_i\cap\Pi_j$, for $j=2,3,4, i\neq j$. Let $\pi$ be the plane spanned by the three points. Take a general point $q$ on $\pi$.  From the fact that all $\Pi_j$ are general, we have that $q$, $p_0:=\pi\cap \Pi_0$ and $p_1:=\pi\cap\Pi_1$ are not on a line.
Then the line through $q$ and $p_0$ is a transversal to $\Pi_0,\Pi_2,\Pi_3,\Pi_4$, so it is in $Q_0$ (and in $\pi$ of course). In the same way, the line  through $q$ and $p_1$ is a transversal to $\Pi_1,\Pi_2,\Pi_3,\Pi_4$, so it is in $Q_1$, thus $q$ is in $Q_0\cap Q_1$.

To prove that $M_4\not\subset T_4$, take a projection from a given point (not in $\Pi_2,\Pi_3,\Pi_4$) to a general hyperplane. Then the intersection of the images of $\Pi_2,\Pi_3,\Pi_4$is either a point - and then there is only one transversal to $\Pi_2,\Pi_3,\Pi_4$ through this point - or this intersection is a line,
and then we have a plane of transversals from our point. From this construction it follows that we may have at most a plane of transversals to $\Pi_2,\Pi_3,\Pi_4$. As $\Pi_1, \Pi_0$ are general, the generic transversal on $\pi$ is not transversal to $\Pi_1, \Pi_0$.

\end{example}

\begin{remark} Snyder and Rusk in \cite{Snyder2}
	  assert that
	$\deg(R_n)=\frac{(n+1)(n-2)}{2}$  and that $\deg(M_n)=\frac{(n-2)(n-3)}{2}$ We plan
a future paper explaining these results and showing also precisely that the inverse of a Veneroni transformation is always a Veneroni.
\end{remark}

\bigskip
\noindent
{\bf Acknowledgements}: {\small Farnik was partially supported by National Science Centre, Poland, grant 2018 /28/C/ST1/00339,
Harbourne was partially supported by Simons Foundation grant \#524858.
 Szemberg was partially supported by National Science Centre grant 2018/30/M/ST1/00148.
Harbourne and Tutaj-Gasi\'nska were partially supported by National Science Centre grant 2017/ 26/M/ST1/00707.
Harbourne and Tutaj-Gasi\'nska thank the Pedagogical University of Cracow, the Jagiellonian University
and the University of Nebraska for hosting reciprocal visits by Harbourne and Tutaj-Gasi\'nska
when some of the work on this paper was done.}

{\footnotesize \noindent
Marcin Dumnicki, Halszka Tutaj-Gasi\'nska\\
	Faculty of Mathematics and Computer Science, Jagiellonian University\\
{\L}ojasiewicza 6, PL-30-348 Krak\'ow, Poland\\
Marcin.Dumnicki@im.uj.edu.pl\\
Halszka.Tutaj@im.uj.edu.pl
\\
\\
\noindent
{\L}ucja Farnik, Tomasz Szemberg\\
Department of Mathematics, Pedagogical University Cracow\\
Podchor\c{a}\.zych 2, PL-30-084 Krak\'ow, Poland\\
lucja.farnik@gmail.com\\
Tomasz.Szemberg@gmail.com
\\
\\
\noindent
Brian Harbourne\\
Department of Mathematics,	University of Nebraska\\
	Lincoln, NE 68588-0130 USA\\
	bharbourne1@unl.edu

}

\end{document}